\documentclass[9pt]{article}

 \usepackage{amsfonts}
\usepackage{amssymb}
\usepackage{amsmath}
\usepackage{amsthm}
\usepackage{latexsym}
\usepackage{graphicx}
\usepackage{layout}
\usepackage{geometry}

\newtheorem{theorem}{Theorem}[section]

\newtheorem{lemma}[theorem]{Lemma}

\newtheorem{remark}[theorem]{Remark}

\renewcommand{\(}{\left(}
\renewcommand{\)}{\right)}
\renewcommand{\[}{\left[}
\renewcommand{\]}{\right]}
\newcommand{\<}{\left\langle}
\renewcommand{\>}{\right\rangle}

 \newcommand{\rr}{ \mathbb{R}}

\begin{document}
\title{Non degeneracy of critical points of the  Robin  function with respect to deformations of the domain}
\author{Anna Maria Micheletti\thanks{
{\sc Dipartimento di Matematica Applicata "U.Dini", Universit\`a di Pisa, Via Filippo Buonarroti, 1
56127 Pisa, Italy}
\textit{E-mail address}: a.micheletti@dma.unipi.it }\and  Angela
Pistoia
\thanks{{\sc  Dipartimento di Scienze di Base e Applicate per l'Ingegneria, Universit\`a
di Roma ``La Sapienza", via A. Scarpa 16, 00161 Roma, Italy.}
\textit{E-mail address}:  pistoia@dmmm.uniroma1.it }}
\maketitle
\date{}
\begin{abstract}
We show a result of genericity for non degenerate critical points of the Robin  function with respect to deformations of the domain.



\end{abstract}

\section{Introduction}

Let   $\Omega$ be a smooth bounded domain in $\rr^N,$  $N\ge2. $
The Green  function of the Laplace operator vanishing at the boundary $\partial\Omega$ is of the form
\begin{equation}\label{green}
G_y(x)={1\over\omega_N}\[\Gamma_y(x)-H_y(x)\],\ x,y\in\Omega,
\end{equation}
with
where $\omega_N$ denotes the surface area of the unit sphere in $\rr^N.$
The singular part $\Gamma_y$ is given by $\Gamma_y(x)=\Gamma\(|x-y|\)$
\begin{equation}\label{Gamma}
 \Gamma\(|x-y|\)=
 -\ln|x-y|\  \hbox{if}\ N=2\quad \hbox{and}\quad
  \Gamma\(|x-y|\)={1\over N-2}|x-y|^{2-N}\   \hbox{if}\ N\ge3.
\end{equation}

The regular part $H_y$ is a harmonic function with the same boundary value as the singular part, i.e. for any $y\in\Omega$
\begin{equation}\label{Gamma2}
\left\{\begin{aligned}
&\Delta_x H_y(x)=0 \ &\hbox{if}\ x\in\Omega\\
  &H_y(x)=\Gamma_y(x) \ &\hbox{if}\ x\in\partial\Omega.\\
\end{aligned}\right.
\end{equation}

The {\em Robin function } of $\Omega$ is   defined by
$t(x):=t^\Omega(x):=H_x(x),\ x\in\Omega.
$

This function plays an important role in various fields of the mathematics, e.g., geometric function
theory, capacity theory, concentration problems (see \cite{BF} and the references therein).

In particular,
existence    and uniqueness of solutions of some critical problems
is strictly dependent on the  non degeneracy of critical points of the Robin function (see, for example, \cite{BLR,GJP,P,R}.
Non degenerate critical points of the Robin function plays also a crucial role in
studying existence and uniqueness of solutions of the Gelfand's problem  (see for example \cite{BP,DKM,EGP,GG,MW}).

As far as we know, the only results about non degeneracy of critical points of the Robin function are in \cite{CF} and \cite{G}. In \cite{CF} the
authors show that  the Robin function
of a smooth bounded and convex domain of $\rr^2$ has a unique critical point which is non degenerate. In \cite{G} the
author   proves that   the origin is a non degenerate critical point of the Robin function
of a smooth bounded domain of $\rr^N$ which is symmetric with respect
to the origin  and convex  in any directions $x_1,\dots,x_N.$

  Here we prove that for {\em most domains  the critical points of the Robin function are non degenerate.}

 \medskip

 Let $\Omega\subset\rr^N$ be a domain of class $C^k$ with $k\ge4$ and $N\ge2.$ We consider the domain
 $\Omega_\theta:=(I+\theta)\Omega$ given by the deformation $I+\theta.$ Here $I$ is the identity map on $\rr^N.$

We are interested in studying the non degeneracy of the critical points of the Robin function of the
domain $\Omega_\theta$ with respect to the parameter $\theta.$

Let $\mathfrak{E}^k $ be the vector space of all the $C^k$ applications $\theta:\rr^N\to\rr^N$ such that
\begin{equation}\label{1}\|\theta\|_k:=\sup\limits_{x\in\rr^N}\max\limits_{0\le|\alpha|\le k}\left|{\partial ^\alpha\theta_i(x)}\over\partial {x_1}^{\alpha_1}\dots\partial {x_N}^{\alpha_N}\right|<+\infty.\end{equation}
 $\mathfrak{E}^k $ is a Banach space equipped with the norm $\|\cdot\|_k.$
Let $\mathfrak{B}_\rho:=\left\{\theta\in \mathfrak{E}^k \ :\ \|\theta\|_k\le\rho\right\}$
be the ball in $\mathfrak{E}^k$ centered at $0$ with radius $\rho.$
We will prove the following result.
\begin{theorem}\label{main}
The set
 $\mathfrak{A}:=\left\{\theta\in \mathfrak{B}_\rho\ :\ \right.$  all the critical points of the  Robin function  of the domain $\Omega_\theta$ are non degenerate $\left.\right\}$
 is a residual (hence dense) subset of $\mathfrak{B}_\rho,$ provided $\rho$ is small enough.\end{theorem}

To get Theorem \ref{main} we use an abstract transversality theorem previously used by Quinn \cite{Q}, Saut and Temam \cite{ST} and Uhlenbeck \cite{U}.
The strategy in our work is similar to the one used by Saut and Temam in \cite{ST} to get some generic property with respect to the domain of the solutions
to certain semilinear elliptic equations. In our case  we need some new delicate estimates which involve   the derivative of
Robin function with respect to the variation of the domain.

The paper is organized as follows. In Section 2 we set the problem and  we prove the main result.  All the technical results are  proved in Section 3
and in Section 4.

\section{Setting of the problem and proof of the main result}

First of all let us recall some useful properties of the Robin function (see \cite{BF}).

\begin{remark}\label{pag2}
If $\Omega$ is of class $C^{2,\alpha}$ then the Robin function $t\in C^{2,\alpha}(\overline\Omega)$ and it holds
\begin{align}\label{2}
&\nabla t(x)=2\nabla_x H_y(x)_{|_{y=x}}\quad\hbox{and}\quad{\partial^2 t\over\partial x_i\partial x_j}(x)=4{\partial^2 H_y\over\partial x_i\partial x_j}(x)_{|_{y=x}}
\end{align}

\end{remark}

Given  $\Omega\subset\rr^N,$  $N\ge2 $ of class $C^k$ with $k\ge3,$ we consider the domain $\Omega_\theta:=(I+\theta)\Omega$ with $\theta\in
\mathfrak{B}_\rho.$ It is well known that we can choose
$\rho$ positive and small enough such that if $\theta\in \mathfrak{B}_\rho$ then the map $I+\theta:\overline\Omega\to(I+\theta)\overline\Omega$ is a diffeomorphism of class $C^k.$   We set
$I+\gamma= (I+\theta)^{-1}.$

\begin{remark}\label{pag2bis}
Since, by definition $(I+\theta)\circ(I+\gamma) =I$ we have that $\gamma(z)=-\theta\(z+\gamma(z)\).$ Moreover, it holds
\begin{equation}\label{star}
\[I+\theta'\(z+\gamma(z)\)\] \(h+\gamma'(z)(h)\)=h\quad\forall\ h\in \rr^N.\end{equation}
Then we have
$$\[I+\theta'\(z+\gamma(z)\)\] \circ \gamma'(z)=-\theta'\(z+\gamma(z)\)$$
which implies
\begin{equation}\label{star2}
\gamma'(z)=-\[I+\theta'\(x\)\]^{-1}\circ\theta'(x)=\sum\limits_{i\ge0}(-1)^{i+1}\[\theta'(x)\]^{i+1}\quad\hbox{where} \ x=z+\gamma(z).\end{equation}
Moreover by \eqref{star}
$$\[I+\theta'\(z+\gamma(z)\)\] \circ \gamma''(z)(h)(k)=-\theta''\(z+\gamma(z)\)\(h+\gamma'(z)(h)\)\(k+\gamma'(z)(k)\)\quad\forall\ h,k\in \rr^N$$
Then we  have if $x=z+\gamma(z)$
\begin{align*} &\gamma''(z)(h)(k)= \sum\limits_{i\ge0}(-1)^{i  }\[\theta'(x)\]^{i  }\[-\theta''\(x\)\(h+\gamma'(z)(h)\)\(k+\gamma'(z)(k)\)\]\\
&=-\theta''\(x\)\(h+\gamma'(z)(h)\)\(k+\gamma'(z)(k)\) -\sum\limits_{i\ge1}(-1)^{i  }\[\theta'(x)\]^{i  }\[ \theta''\(x\)\(h+\gamma'(z)(h)\)\(k+\gamma'(z)(k)\)\]
 .\\ \end{align*}

\end{remark}

Given $y\in\Omega_\theta$,
we consider the unique function $v_y^\theta$ solution of the problem
\begin{equation}\label{4}
\left\{\begin{aligned}
&\Delta_z v_y^\theta(z)=0 \ &\hbox{if}\ z\in\Omega_\theta\\
  &v_y^\theta(z)=\Gamma\(|z-y|\) \ &\hbox{if}\ z\in\partial\Omega_\theta.\\
\end{aligned}\right.
\end{equation}

More precisely, $v_y^\theta$ is the regular part of the Green's function of the domain $\Omega_\theta.$ We have that
$v_y^\theta\in C^{2,\alpha}(\overline\Omega_\theta)$ because $
k\ge3.$ If $\xi\in\Omega$ is such that $\xi+\theta(\xi)=y,$ we define the function $\widetilde v^\theta_\xi\in C^{2,\alpha}(\overline\Omega )$ by

\begin{equation}\label{5}
\widetilde v^\theta_\xi(x):=v_y^\theta\(x+\theta(x)\)=v^\theta_y(z).
\end{equation}

The function $\widetilde v^\theta_\xi$ is the unique solution of the following problem
\begin{equation}\label{6}
\left\{\begin{aligned}
 &\sum\limits_{i,j,s=1}^N\left.{\partial ^2\widetilde v^\theta_\xi\over\partial x_i\partial x_j} \right|_x\[\delta_{is}+
\left.{\partial\gamma_i\over\partial z_s} \right|_{x+\theta(x)} \]\[\delta_{js}+
\left.{\partial\gamma_j\over\partial z_s} \right|_{x+\theta(x)} \] +\sum\limits_{ j,s=1}^N\left.{\partial  \widetilde v^\theta_\xi\over\partial x_ j}\right|_x\left.{\partial ^2\gamma_j\over\partial z_s^2}\right|_{x+\theta(x)}  =0 \ &\hbox{if}\ x\in\Omega \\
  &\widetilde v^\theta_\xi(x)=\Gamma\(|x-\xi+\theta(x)-\theta(\xi)|\) \ &\hbox{if}\ x\in\partial\Omega .\\
\end{aligned}\right.
\end{equation}

When $\theta=0$ we   obviously have that $\widetilde v^0_\xi$ is the unique solution of the problem
\begin{equation}\label{6bis}
\left\{\begin{aligned}
 &\Delta_x \widetilde v^0 =0 \ &\hbox{if}\ x\in\Omega \\
  &\widetilde v^0_\xi(x)=\Gamma\(|x-\xi|\) \ &\hbox{if}\ x\in\partial\Omega .\\
\end{aligned}\right.
\end{equation}

\begin{remark}\label{pag3}
It is easy to see that there exists a $C^3-$extension of the function $\Gamma(|z-y|)$ for $z\in\partial\Omega_\theta$ on the domain $\Omega_\theta$, $z\to\Gamma(|z-y|)\chi_d(|z-y|).$ Here the smooth cut off function $\chi_d$ is such that
$$\chi_d(s)=0\ \hbox{if}\ 0<s<d,\ \chi_d(s)=1\ \hbox{if}\ s>2d,\ |\chi'(s)|<{c\over d},\ |\chi''(s)|<{c\over d^2}$$
where  $d=\mathrm{dist}(y,\partial\Omega_\theta)/3,$ for some constant $c>1.$

Since $\widetilde v^\theta_\xi$ solves \eqref{6}, by maximum principle and standard elliptic regularity theory (see Theorem 6.6, \cite{GT}) we get
$$\|\widetilde v^\theta_\xi\|_{C^{2,\alpha}(\overline\Omega)}\le c\[\|\widetilde v^\theta_\xi\|_{C^{0}(\overline\Omega)}+\|\varphi\|_{C^{2,\alpha}(\overline\Omega)}\]
\le c\[\sup\limits_{x\in\partial\Omega}\Gamma\(|x-\xi+\theta(x)-\theta(\xi)|\)+\|\varphi\|_{C^{2,\alpha}(\overline\Omega)}\],$$
where
$$\varphi(x):=\Gamma\(|x-\xi+\theta(x)-\theta(\xi)|\)\chi_d\(|x-\xi+\theta(x)-\theta(\xi)|\).$$

It is important to point out that  by standard regularity theory (see Theorem 6.6, \cite{GT}) we also   get that
 $\widetilde v^\theta_\xi\in C^{3,\alpha}(\overline\Omega) $ if $k\ge4.$

\end{remark}
Let us establish some properties of the function $\widetilde v^\theta_\xi.$

It is useful to point out that when $\theta=0$, for any $p=1,\dots,N$ the function $w^0_p:={\partial\over\partial x_p}\widetilde v^0_\xi$ is the unique solution of the following problem
\begin{equation}\label{7bis}
\left\{\begin{aligned}
 &\Delta_x   w^0_p   =0 \qquad \hbox{if}\ x\in\Omega \\
  &w_p^ 0(x)= {x_p-\xi_p\over |x -\xi |^N}
   \qquad \hbox{if}\ x\in\partial\Omega .\\
\end{aligned}\right.
\end{equation}

We are interested in studying the non degeneracy of the critical points of the Robin function of the domain
$\Omega_\theta$, namely by \eqref{2} and \eqref{4} the points $z\in\Omega_\theta$ such that
\begin{equation}\label{9}
0=\nabla_zt^{\Omega_\theta}(z)=2\nabla_zv^\theta_y(z)|_{y=z}.
\end{equation}

This is equivalent to study the non degeneracy of $x\in\Omega$ such that $0=\nabla_x\widetilde v^\theta_\xi(x)|_{\xi=x}.$
Thus, we are led to consider the map $F:\Omega\times \mathfrak{B}_\rho\to\rr^N$
defined by
\begin{equation}\label{10}
F(x,\theta):=\nabla_x\widetilde v^\theta_\xi(x)|_{\xi=x}.
\end{equation}
By Remark \ref{pag3} and Lemma \ref{lemma2} $F$ is a $C^1-$map.

 We shall apply the following abstract transversality theorem to the map $F$ (see \cite{Q,ST,U}).

 \begin{theorem}\label{tran}
 Let $X,Y,Z$ be three Banach spaces and $U\subset X,$ $V\subset Y$ open subsets.
 Let $F:U\times V\to Z$ be a $C^\alpha-$map with $\alpha\ge1.$ Assume that

 \begin{itemize}
 \item[i)] for any $y\in V$, $F(\cdot,y):U\to Z$ is a Fredholm map of index $l$ with $l\le\alpha;$
 \item[ii)] $0$ is a regular value of $F$, i.e. the operator $F'(x_0,y_0):X\times Y\to Z$ is onto at any point $(x_0,y_0)$ such that $F(x_0,y_0)=0;$
 \item[iii)] the map  $\pi\circ i:F^{-1}(0)\to Y$ is $\sigma-$proper, i.e.  $F^{-1}(0)=\cup_{s=1}^{+\infty} C_s$
 where $C_s$ is a closed set and the restriction $\pi\circ i_{|_{C_s}}$ is proper for any $s$; here $i:F^{-1}(0)\to Y$ is the canonical embedding and $\pi:X\times Y\to Y$ is the projection.
 \end{itemize}

 Then the set
$\Theta:=\left\{y\in V\ :\ 0\ \hbox{is  a regular value of } F(\cdot,y)\right\}$
 is a  residual subset of $V$, i.e. $V\setminus \Theta$ is a countable union of closet subsets without interior points.

\end{theorem}

\medskip

\begin{proof}[Proof of the main result]

We are going to apply the transversality theorem \ref{tran} to the map $F$ defined by \eqref{10}. In this case we have
$X=Z=\rr^{N},$ $Y=\mathfrak{E}^k,$ $U=\Omega\subset\rr^{N }$ and $V=\mathfrak{B}_\rho\subset \mathfrak{E}^k,$
where   $\rho$ is small enough. Since $X=Z$ is a finite dimensional space, it is easy to check that for any $\theta\in \mathfrak{B}_\rho,$ the map $x\to F(x,\theta)$ is a Fredholm map of index $0$ and then assumption i) holds.
  As far as it concerns assumption  iii), we have that
$$F^{-1}(0)= \cup_{s=1}^{+\infty} C_s,\ \hbox{where}\ C_s:=\left\{ {\Omega_s}\times \overline  {\mathfrak{B}  _{\rho-{1\over s}}} \right\}\cap F^{-1}(0) \ \hbox{and}\
 \Omega_s:=\left\{x\in\Omega\ :\ \mathrm{dist}(x,\partial\Omega)\le 1/s\right\}.$$
Using the compactness of $\Omega_s,$ we can show  that the restriction
$\pi\circ i_{|_{C_s}}$ is proper, namely if the sequence $(\theta_n)\subset \overline  {\mathfrak{B}  _{\rho-{1\over s}}} $ converges to $\psi_0$ and the sequence $(x_n)\subset \Omega_s$ is such that
$F(x_n,\theta_n)=0$ then there exists a subsequence of $(x_n)$ which converges to $x_0\in \Omega_s $ and $F(x_0,\psi_0)=0.$

To prove that assumption ii) holds we will show in Lemma \ref{lemma42} that if  $(\bar x,\bar \theta)\in\Omega\times\mathfrak{B}_\rho$ is such that $F(\bar x,\bar \theta)=\nabla_x\widetilde v^{\bar\theta}_{\xi}(\bar x)|_{\xi=\bar x}=0$ the map $F'_\theta(\bar x,\bar \theta):\mathfrak{E}^k\to\rr^N$  defined by
 $\theta\to D_\theta  \nabla_x\widetilde v^\theta_{\bar x}(  x)|_{\theta=\bar\theta, x=\bar x}[\theta] $ is surjective.

Finally, we can apply the transversality theorem \ref{tran} and we get that
the set
\begin{align*}
 \mathfrak{A}:=&\left\{\theta\in \mathfrak{B}_\rho \ :\ F'_x(x,\theta):\rr^{N }\to \rr^{N } \ \hbox{is invertible at any point}\ (x,\theta)\  \hbox{such that}\ F(x,\theta )=0\right\}\nonumber\\
=&\left\{\theta\in \mathfrak{B}_\rho \ :\  \hbox{the critical points of the Robin function}\ \hbox{ of the domain $ \Omega_\theta$ are nondegenerate}
  \right\}\end{align*}
is a residual, and hence dense, subset of $\mathfrak{B}_\rho.$

\end{proof}

\section{$0$ is a regular value of $F$}\label{sec4}
In this section we show that $0$ is a regular value of the map  $F$ defined by \eqref{10}.
\begin{lemma}\label{lemma42}
The map $\theta\to F'_\theta(\bar x,\bar\theta)[\theta]$ is onto on $\rr^N$ for any $(\bar x,\bar\theta)\in \Omega\times \mathfrak{B}_\rho$ such that $F(\bar x,\bar\theta)=0.$
\end{lemma}
\begin{proof}
Let us fix $(\bar x,\bar\theta)\in \Omega\times \mathfrak{B}_\rho$ such that $F(\bar x,\bar\theta)=0.$
We want to show that given $e^{(1)},\dots,e^{(N)}$ the canonical base in $\rr^N$, for any $i=1,\dots,N$ there exists $\theta\in\mathfrak{E}^k$ such that $F'_\theta(\bar x,\bar\theta)[\theta]=e^{(i)}.$
We point out that the ontoness of the map $\theta\to F'_\theta(\bar x,\bar\theta)[\theta]$ is invariant with respect
to the change of variables $\eta=(I+\bar\theta)(x).$
We have that
\begin{equation}\label{4.1}
F'_\theta(\bar x,\bar\theta)[\theta]=\({\partial\over\partial x_1}D_\theta\widetilde v^{\bar\theta}_{\bar x}[\theta](\bar x),\dots,{\partial\over\partial x_N}D_\theta\widetilde v^{\bar\theta}_{\bar x}[\theta](\bar x)\)
\end{equation}
because  ${\partial\over\partial x_p}D_\theta\widetilde v^{\bar\theta}_{\bar x}[\theta](\bar x)=D_\theta{\partial\over\partial x_p}\widetilde v^{\bar\theta}_{\bar x}[\theta](\bar x) $ as it is easy to verify.

Let $\bar\eta=\bar x+\bar\theta(\bar x)\in\Omega_{\bar \theta}.$ By \eqref{4}, \eqref{5}, \eqref{6} and Lemma \ref{lemma2}, we deduce that
$v_{\bar\eta}^{\bar\theta}(\eta)=v_{\bar\eta}^{\bar\theta}(x+\bar\theta (x))=\widetilde v^{\bar\theta}_{\bar x}(x)$
is the unique solution of
$$\left\{\begin{aligned}
 &\Delta_\eta  v_{\bar\eta}^{\bar\theta}
 =0 \qquad \hbox{if}\ \eta\in\Omega_{\bar \theta} \\
  &v_{\bar\eta}^{\bar\theta}(\eta)= \Gamma\( |\eta -\bar\eta|\)
   \qquad \hbox{if}\ \eta\in\partial\Omega _{\bar \theta}.\\
\end{aligned}\right.
$$
and
$\nabla_\eta v_{\bar\eta}^{\bar\theta}(\eta)|_{\eta=\bar\eta}=0.$

We consider the deformation $I+\bar\theta+\theta=(I+\alpha)(I+\bar\theta),$ where $\theta=\alpha(I+\bar\theta)$
and the domain $\(I+\bar\theta+\theta\)\Omega=(I+\alpha)(I+\bar\theta)\Omega.$
We set
\begin{equation}\label{4.2}
 \bar\eta:=(I+\bar\theta)\bar x\quad\hbox{and}\quad \bar z:=(I+\alpha)\bar\eta.
 \end{equation}
Let $v_{\bar z}^{\bar \theta+\theta}$ be the unique solution of
\begin{equation}\label{4.3}
\left\{\begin{aligned}
 &\Delta_z  w(z)
 =0 \qquad \hbox{if}\ z\in\Omega_{\bar \theta+\theta} \\
  &w(z)= \Gamma\( |z -\bar z|\)
   \qquad \hbox{if}\ z\in\partial\Omega _{\bar \theta+\theta}.\\
\end{aligned}\right.
\end{equation}
 Then we set
 $$v_{\bar z}^{\bar \theta+\theta}(z)=v_{\bar z}^{\bar \theta+\theta}\(\eta+\alpha(\eta)\)=
\hat v_{\bar \eta}^{\bar \theta+\theta}(\eta)=\hat v_{\bar \eta}^{\bar \theta+\theta}\(x+\bar \theta(x)\)=
\widetilde v_{\bar x}^{\bar\theta+\theta}(x).$$
We immediately obtain that
\begin{equation}\label{4.4}
D_\theta\hat v_{\bar \eta}^{\bar \theta+\theta}|_{\theta=0}[\beta](\eta)=D_\theta\widetilde v_{\bar x}^{\bar\theta+\theta} |_{\theta=0}[\beta](x)\ \hbox{with}\ \eta=x+\bar\theta(x).
\end{equation}

By \eqref{4.4} we have that given $\theta^{(1)},\dots,\theta^{(N)}$ the $N$ vectors
$$\nabla_x D_\theta\widetilde v_{\bar x}^{\bar\theta+\theta} |_{\theta=0}[\theta^{(1)}](\bar x),\dots, \nabla_x D_\theta\widetilde v_{\bar x}^{\bar\theta+\theta} |_{\theta=0}[\theta^{(N)}](\bar x)$$
are linearly independent if and only if the $N$ vectors
$$\nabla_x D_\theta\hat v_{\bar \eta}^{\bar \theta+\theta}|_{\theta=0}[\theta^{(1)}](\bar \eta),\dots, \nabla_x D_\theta\hat v_{\bar \eta}^{\bar \theta+\theta} |_{\theta=0}[\theta^{(N)}](\bar \eta) $$
are linearly independent.

At this stage our aim is to find $\theta^{(1)},\dots,\theta^{(N)}$ so that the $N$ vectors
$$\nabla_x D_\theta\hat v_{\bar \eta}^{\bar \theta+\theta}|_{\theta=0}[\theta^{(1)}](\bar \eta),\dots, \nabla_x D_\theta\hat v_{\bar \eta}^{\bar \theta+\theta} |_{\theta=0}[\theta^{(N)}](\bar \eta) $$
are linearly independent. First of all we point out that by Lemma \ref{p21} the function
$w_{\bar\eta}^{\bar\theta}[\alpha](\cdot):=D_\theta\hat v_{\bar \eta}^{\bar \theta+\theta}|_{\theta=0}[\theta](\cdot)$ is the unique solution of the
problem
\begin{equation}\label{4.5}
\left\{\begin{aligned}
 & \Delta_\eta  w-\sum\limits_{i,j=1}^N{\partial^2 \hat v_{\bar\eta}^{\bar\theta}\over
 \partial\eta_i\partial\eta_j}(\eta)\[{\partial \alpha_j\over\partial\eta_i}(\eta)+{\partial \alpha_i\over\partial\eta_j}(\eta)
 \]-\sum\limits_{j=1}^N{\partial  \hat v_{\bar\eta}^{\bar\theta}\over
 \partial\eta_j}(\eta)\Delta_\eta \alpha_j (\eta)=0 \ &\hbox{if}\ \eta\in\Omega _{\bar\theta}\\
  &w(\eta)=\sum\limits_{i=1}^N {\eta_i-\bar\eta_i\over |\eta-\bar\eta|^N}\(\alpha_i(\eta)-\alpha_i(\bar\eta)\) \ &\hbox{if}\ \eta\in\partial\Omega _{\bar\theta}.\\
\end{aligned}\right.
\end{equation}
Here $\alpha=\theta(I+\bar\theta)^{-1}$ and $\hat v_{\bar\eta}^{\bar\theta}$ is the unique solution of
$$\left\{\begin{aligned}
 & \Delta_\eta  \hat v_{\bar\eta}^{\bar\theta}=0 \ &\hbox{if}\ \eta\in\Omega _{\bar\theta}\\
  &\hat v_{\bar\eta}^{\bar\theta}(\eta)=\Gamma\(|\eta -\bar\eta|\) \ &\hbox{if}\ \eta\in\partial\Omega _{\bar\theta}.\\
\end{aligned}\right.
$$
We remark that by standard regularity theory (see also Remark \ref{pag3}) it follows
\begin{equation}\label{4.5bis}
 \| \hat v_{\bar\eta}^{\bar\theta}\|_{C^{3}(\overline \Omega _{\bar\theta} )}\le c(\bar\theta,\bar\eta),
\end{equation}
for some positive constant depending only on $\bar \theta$ and $\bar\eta.$

Moreover, we also get that  the function $\eta\to{\partial\over\partial\eta_p}D_\theta\hat v_{\bar\eta}^{\bar\theta+\theta}|_{\theta=0}[\theta](\eta)={\partial\over\partial\eta_p}w_{\bar\eta}^{\bar\theta}[\alpha](\eta)$ for $p=1,\dots,N$ is the unique solution
of the problem
\begin{equation}\label{4.6}
\left\{\begin{aligned}
 & \Delta_\eta  {\partial\over\partial\eta_p} w_{\bar\eta}^{\bar\theta}[\alpha](\eta)-{\partial\over\partial\eta_p}\left\{\sum\limits_{i,j=1}^N{\partial^2 \hat v_{\bar\eta}^{\bar\theta}\over
 \partial\eta_i\partial\eta_j}(\eta)\[{\partial \alpha_j\over\partial\eta_i}(\eta)+{\partial \alpha_i\over\partial\eta_j}(\eta)
 \]\right\}  -{\partial\over\partial\eta_p}\left\{\sum\limits_{j=1}^N{\partial  \hat v_{\bar\eta}^{\bar\theta}\over
 \partial\eta_j}(\eta)\Delta_\eta \alpha_j (\eta)\right\}=0 \  \hbox{if}\ \eta\in\Omega _{\bar\theta}\\
  &{\partial\over\partial\eta_p} w_{\bar\eta}^{\bar\theta}[\alpha](\eta)={\partial\over\partial\eta_p}\left\{\sum\limits_{i=1}^N {\eta_i-\bar\eta_i\over |\eta-\bar\eta|^N}\(\alpha_i(\eta)-\alpha_i(\bar\eta)\) \right\}\  \hbox{if}\ \eta\in\partial\Omega _{\bar\theta}.\\
\end{aligned}\right.
\end{equation}
Therefore we look for   $\alpha^{(1)},\dots,\alpha^{(N)}$ such that the $N$ vectors
$$\nabla_\eta w^{\bar\theta}_{\bar\eta}[\alpha^{(1)}](\bar\eta),\dots,\nabla_\eta w^{\bar\theta}_{\bar\eta}[\alpha^{(N)}](\bar\eta)$$
are linearly independent.
Using the Green's representation formula by \eqref{4.6} we get
\begin{align}\label{4.7}
& {\partial\over\partial\eta_p} w_{\bar\eta}^{\bar\theta}[\alpha](\bar\eta)= \int\limits_{\partial\Omega_{\bar\theta}}{\partial\over\partial\eta_p}\left\{\sum\limits_{i=1}^N {\eta_i-\bar\eta_i\over |\eta-\bar\eta|^N}\(\alpha_i(\eta)-\alpha_i(\bar\eta)\) \right\}{\partial G\over\partial\nu}(\eta,\bar\eta)d\sigma\nonumber \\
 &+\int\limits_{ \Omega_{\bar\theta}} {\partial\over\partial\eta_p}\left\{\sum\limits_{i,j=1}^N{\partial^2 \hat v_{\bar\eta}^{\bar\theta}\over
 \partial\eta_i\partial\eta_j}(\eta)\[{\partial \alpha_j\over\partial\eta_i}(\eta)+{\partial \alpha_i\over\partial\eta_j}(\eta)
 \]- \sum\limits_{j=1}^N{\partial  \hat v_{\bar\eta}^{\bar\theta}\over
 \partial\eta_j}(\eta)\Delta_\eta \alpha_j (\eta)\right\}G(\eta,\bar\eta)d\eta.
 \end{align}
We now choose $\alpha^{(1)}$ so that
 $$\alpha^{(1)}_1( \eta)=|\eta-\bar\eta|^N\chi\(\[\mathrm{dist}(\eta,\partial\Omega_{\bar\theta})\]^a\)\ \hbox{and}\ \alpha^{(1)}_2( \eta)=\dots=\alpha^{(1)}_N( \eta)=0.$$
Since $\partial\Omega_{\bar\theta}$ is smooth, the function $\eta\to \mathrm{dist}(\eta,\partial\Omega_{\bar\theta})$ is of class $C^3$ when $\eta$ is close enough
to the boundary.
 Here  the cut off function $\chi$ is of class $C^3$ and satisfies
\begin{equation}\label{4.8}\chi(s)=1\ \hbox{if}\ s\in(0,\bar\rho),\ \chi(s)=0\ \hbox{if}\ s\in(2\bar\rho,\infty),\ |\chi'(s)|\le {1\over\bar\rho},\ |\chi''(s)|\le {1\over\bar\rho^2},\
 |\chi'''(s)|\le {1\over\bar\rho^3}\end{equation}
 where $\bar\rho>0$ is such that $4\bar \rho\le \mathrm{dist}(\bar\eta,\partial\Omega_{\bar\theta})$ and $\bar\rho$ will be chosen  small enough.
 The positive number $a $ will be chosen $a\ge 4 $ (so that estimate \eqref{4.16} holds).

By the definition of $\alpha^{(1)}$ and \eqref{4.7} we have
\begin{align}\label{4.9}
&  \int\limits_{\partial\Omega_{\bar\theta}}{\partial\over\partial\eta_p}\left\{\sum\limits_{i=1}^N {\eta_i-\bar\eta_i\over |\eta-\bar\eta|^N}\(\alpha_i(\eta)-\alpha_i(\bar\eta)\) \right\}{\partial G\over\partial\nu}(\eta,\bar\eta)d\sigma =\int\limits_{\partial\Omega_{\bar\theta}}{\partial\over\partial\eta_p}  (\eta_1-\bar\eta_1){\partial G\over\partial\nu}(\eta,\bar\eta)d\sigma=\delta_{1p}\int\limits_{\partial\Omega_{\bar\theta}} {\partial G\over\partial\nu}(\eta,\bar\eta)d\sigma \end{align}
Moreover we have
\begin{align}\label{4.10}
&\int\limits_{ \Omega_{\bar\theta}} {\partial\over\partial\eta_p}\left\{\sum\limits_{i,j=1}^N{\partial^2 \hat v_{\bar\eta}^{\bar\theta}\over
 \partial\eta_i\partial\eta_j}(\eta)\[{\partial \alpha_j\over\partial\eta_i}(\eta)+{\partial \alpha_i\over\partial\eta_j}(\eta)
 \]- \sum\limits_{j=1}^N{\partial  \hat v_{\bar\eta}^{\bar\theta}\over
 \partial\eta_j}(\eta)\Delta_\eta \alpha_j (\eta)\right\}G(\eta,\bar\eta)d\eta\nonumber\\ &
=\int\limits_{ \Omega_{\bar\theta}^{\bar\rho}} {\partial\over\partial\eta_p}\left\{2{\partial^2 \hat v_{\bar\eta}^{\bar\theta}\over
 \partial\eta_1^2}(\eta) {\partial \alpha_1^{(1)}\over\partial\eta_1}(\eta) -  {\partial  \hat v_{\bar\eta}^{\bar\theta}\over
 \partial\eta_1}(\eta)\Delta_\eta \alpha_1^{(1)} (\eta)\right\}G(\eta,\bar\eta)d\eta =:\sigma^{(1)}_p(\bar\rho),
 \end{align}
where
 $\Omega_{\bar\theta}^{\bar\rho}:=\left\{\eta \in\Omega_{\bar\theta}\ :\ \mathrm{dist}( \eta,\partial\Omega_{\bar\theta})<2\bar\rho\right\}.$

We now establish an accurate estimate of $\sigma^{(1)}(\bar\rho).$
 By Lemma \ref{lemma41}, proved at the end of this section, for $\bar\rho$ small enough we have that there exists $c_1>0$ such that
\begin{equation}\label{4.11}
|G(\eta,\bar\eta)|\le c_1\bar\rho\ \hbox{for any}\ \eta\in\Omega_{\bar\theta}^{\bar\rho}.
\end{equation}

Moreover, it is easy to check that there exists $c_2>0$ such that for any $t=(t_1,\dots,t_N)$ with $|t|\le3$
\begin{equation}\label{4.13}
\left|{\partial ^t |\eta-\bar\eta|^N\over\partial \eta_1^{t_1}\cdots\partial\eta_N^{t_N}}\right|\le \left\{c_2 \ \hbox{if}\ N\ge3,\   c_2\bar\rho^{-1}\ \hbox{if}\ N=2\right\}\quad \hbox{for any}\ \eta\in\Omega_{\bar\theta}^{\bar\rho}.
\end{equation}

By \eqref{4.10}, \eqref{4.11}, \eqref{4.13} and \eqref{4.5bis} it follows that
\begin{align}\label{4.14}
&\sigma^{(1)}_p(\bar\rho)
 \le c\int\limits_{ \Omega_{\bar\theta}^{\bar\rho}} \left\{ \left| {\partial \alpha_1^{(1)}\over\partial\eta_1}\right|
 +\left| {\partial ^2\alpha_1^{(1)}\over\partial\eta_1\partial\eta_p}\right|+\left| \Delta_\eta \alpha_1^{(1)} \right|+
 \left| {\partial  \over \partial\eta_p}\Delta_\eta \alpha_1^{(1)}\right|
 \right\}d\eta\nonumber\\ &
\le c\int\limits_{ \Omega_{\bar\theta}^{\bar\rho}} \underbrace{\left\{ \left| {\partial  \over\partial\eta_1}\chi\(d^a(\eta)\)\right|
 +\left| {\partial ^2 \over\partial\eta_1\partial\eta_p}\chi\(d^a(\eta)\)\right|+\left| \Delta_\eta \chi\(d^a(\eta)\) \right|+
 \left| {\partial  \over \partial\eta_p}\Delta_\eta \chi\(d^a(\eta)\)\right|
 \right\}}_{A_p(\eta)}d\eta
 \end{align}
where $d^a(\eta):=\[\mathrm{dist}(\eta,\partial\Omega_{\bar\theta})\]^a.$

Let us estimate $A_p(\eta)$ when $\eta\in \Omega_{\bar\theta}^{\bar\rho}.$
We recall that $0\le d(\eta)\le\bar\rho$ since $\eta\in\Omega_{\bar\theta}^{\bar\rho}$ and \eqref{4.8} holds.
By a  simple  calculation of the derivatives of the function $\eta\to\chi\(d^a(\eta)\) $ we easily get   that there exists $c_3>0$ such that
\begin{equation}\label{4.15}
0\le A_p(\eta)\le c_3\(\bar\rho^{a-2}+\bar\rho^{a-3}+\bar\rho^{a-4}+\bar\rho^{2a-4} +\bar\rho^{2a-5}+\bar\rho^{3a-6}\)\ \hbox{for any}\ \eta\in \Omega_{\bar\theta}^{\bar\rho}.
\end{equation}
Then choosing $a\ge4$ we have that there exists $c_4>0$ such that
\begin{equation}\label{4.16}
0\le A_p(\eta)\le c_2 \ \hbox{for any}\ \eta\in \Omega_{\bar\theta}^{\bar\rho}.
\end{equation}

By \eqref{4.10}, \eqref{4.11}, \eqref{4.14} and \eqref{4.15} we deduce that $\lim\limits_{\bar\rho\to0}\sigma^{(1)}_p(\bar\rho)=0.$
Therefore, by \eqref{4.7}, \eqref{4.9}  and \eqref{4.10} we get
$$\nabla_\eta D_\theta\hat v_{\bar\eta}^{\bar\theta+\theta}|_{\theta=0}[\alpha ^{(1)}](\bar\eta)=\nabla_\eta w_{\bar\eta}^{\bar\theta } [\alpha ^{(1)}](\bar\eta)
=\(\sigma_0+\sigma^{(1)}_1(\bar\rho),\sigma^{(1)}_2(\bar\rho),\dots,\sigma^{(1)}_N(\bar\rho)\),$$
where $\sigma_0:=\int\limits_{\partial\Omega_{\bar\theta} }{\partial G\over\partial\nu}(\eta,\bar\eta)d\sigma\not=0.$

In a similar way, for any $q=1,\dots,N$ we can choose $\alpha ^{(q)}$ such that
 $$\alpha^{(q)}_q( \eta)=|\eta-\bar\eta|^N\chi\(\[\mathrm{dist}(\eta,\partial\Omega_{\bar\theta})\]^a\)\quad \hbox{and}\quad \alpha^{(q)}_i( \eta)=0\ \hbox{if}\ i\not=q.$$

Arguing as above, for any $q=1,\dots,N$ we get
$$\nabla_\eta D_\theta\hat v_{\bar\eta}^{\bar\theta+\theta}|_{\theta=0}[\alpha ^{(q)}](\bar\eta)=\nabla_\eta w_{\bar\eta}^{\bar\theta } [\alpha ^{(q)}](\bar\eta)
=\(\sigma^{(q)}_1(\bar\rho),\dots,\underbrace{\sigma_0+\sigma^{(q)}_q(\bar\rho)}_{q-\mathrm{th}},\dots,\sigma^{(q)}_N(\bar\rho)\),$$
where $\lim\limits_{\bar\rho\to0}\sigma^{(q)}_p(\bar\rho)=0 $ for any $p=1,\dots,N.$

Finally, we choose $\bar\rho$ small enough so that the $N$ vectors $\nabla_\eta w_{\bar\eta}^{\bar\theta } [\alpha ^{(1)}](\bar\eta),\dots,ù\nabla_\eta w_{\bar\eta}^{\bar\theta } [\alpha ^{(N)}](\bar\eta)$ are linearly independent and the claim follows.
\end{proof}

Next, we prove Lemma \ref{lemma41} used in the proof of Lemma \ref{lemma42}.
\begin{lemma}\label{lemma41}
Given $y\in\Omega,$ there exist $\tau_0>0$ and $c_1>0$ such that for any $\tau\in(0,\tau_0)$
$$|G(x,y)|\le c_1 \tau\quad \forall\ x\in\Omega_\tau:=\left\{x\in\Omega\ :\ \mathrm{dist}(x,\partial\Omega)\le\tau\right\}.$$
\end{lemma}
\begin{proof}
Let us fix $y\in\Omega.$
First of all, if $\tau$ is small enough, for any $x\in\Omega_\tau$ there exists a unique $p_x\in\partial \Omega$ such that
\begin{equation}\label{1ss}
\mathrm{dist}(x,\partial\Omega)=|x-p_x|\le\tau.\end{equation}
By mean value theorem we get for some $t\in(0,1)$
$$G(x,y)=G(x,y)-G(p_x,y)=\<\nabla_x G(tx+(1-t)p_x,y),x-p_x\>.$$
Therefore, taking into account that \eqref{1ss} holds and also that $tx+(1-t)p_x\in\Omega_\tau$ for any $ x\in\Omega_\tau,$
we get
$$|G(x,y)|\le \tau\max\limits_{x\in\Omega_\tau}|\nabla_x G(x,y)|.$$
The claim will follow if we prove that
\begin{equation}\label{2ss}
\max\limits_{x\in\Omega_\tau}|\nabla_x G(x,y)|\le c(y),\end{equation}
for some positive constant $c$ depending on $y.$

Let us  recall that (see \eqref{green})
$G(x,y)=\gamma\[\Gamma(|x-y|)-H(x,y)\].$
If we choose $\tau<{\mathrm{dist}(y,\partial\Omega)\over2}$ then
$$|x-y|\ge \mathrm{dist}(y,\partial\Omega)-\mathrm{dist}(x,\partial\Omega)\ge{\mathrm{dist}(y,\partial\Omega)\over2}$$
and so by the expression of $\Gamma$ in \eqref{Gamma} we get
\begin{equation}\label{3ss}
\max\limits_{x\in\Omega_\tau}|\nabla_x \Gamma(x,y)|\le c(y),\end{equation}
for some positive constant $c$ depending on $y.$
 Moreover, by \eqref{Gamma2} and by standard regularity theory (see Remark \ref{pag3}),
we also have that
\begin{equation}\label{4ss}
\max\limits_{x\in\Omega_\tau}|\nabla_x H(x,y)|\le c(y),\end{equation}
for some positive constant $c$ depending on $y.$
Finally, by \eqref{3ss} and \eqref{4ss} and \eqref{green}, we get \eqref{2ss} and so the claim is proved.

\end{proof}

\section{The dependence on $\theta$ of $\widetilde v^\theta_\xi$ and $\nabla_x \widetilde v^\theta_\xi$}\label{sec5}

In the following we calculate the Frechet derivative with respect to $\theta$  of $\widetilde v^\theta_\xi$ and ${\partial\over\partial x_p}\widetilde v^\theta_\xi$ for $p=1,\dots,N.$ Moreover, we prove that the map $\theta\to {\partial\over\partial x_p}\widetilde v^\theta_\xi$ is of class $C^1$ for any $p=1,\dots,N.$

\begin{lemma}\label{p21}
For any $\xi\in\Omega  $  the map $T:\mathfrak{B}_\rho\to C^{2,\alpha}(\overline\Omega)$ defined by $T(\theta)=\widetilde v^\theta_\xi$ is of class $C^1.$ Moreover
$$T'_\theta(0)[\theta]=D_\theta\widetilde v^\theta_\xi|_{\theta=0}[\theta]=u[\theta]$$
is the unique solution of the problem
\begin{equation}\label{1s}
\left\{\begin{aligned}
 & \Delta_x u[\theta]+\sum\limits_{i,j=1}^N{\partial^2 \widetilde v^0_\xi\over
 \partial x_i\partial x_j} \[{\partial \theta_j\over\partial x_i} +{\partial \theta_i\over\partial x_j}
 \]-\sum\limits_{j=1}^N{\partial  \widetilde v^0_\xi\over
 \partial x_j} \Delta_x \theta_j  =0 \ &\hbox{if}\ x\in\Omega  \\
  & u[\theta](x)=-\sum\limits_{i=1}^N {x_i-\xi_i\over |x-\xi|^N}\(\theta_i(x)-\theta_i(\xi)\) \ &\hbox{if}\ x\in\partial\Omega  .\\
\end{aligned}\right.
\end{equation}

\end{lemma}
\begin{proof}
First, we prove that the Gateaux derivative of the map $T$ at $0$ is the unique solution of the problem \eqref{1s}.
It holds
\begin{equation}\label{5s}
\left\{\begin{aligned}
 & \Delta_x \({\widetilde v^{t\theta}_\xi-\widetilde v^0_\xi\over t}\)+\sum\limits_{i,j=1}^N{\partial^2 \over
 \partial x_i\partial x_j} \(\widetilde v^\theta_\xi-\widetilde v^0_\xi\){1\over t} \[{\partial \gamma_j^t\over\partial z_i} +{\partial \gamma_i^t\over\partial z_j}+\sum\limits_{s=1}^N{\partial \gamma_j^t\over\partial z_s}{\partial \gamma_i^t\over\partial z_s}
 \]\nonumber\\ &\hskip3truecm +\sum\limits_{j=1}^N{\partial   \over
 \partial x_j}\(\widetilde v^\theta_\xi-\widetilde v^0_\xi\) {1\over t} \sum\limits_{s=1}^N{\partial ^2\gamma_j^t\over\partial z_s^2}+f^t =0 \ &\hbox{if}\ x\in\Omega  \\
  & \({\widetilde v^{t\theta}_\xi-\widetilde v^0_\xi\over t}\)(x)={\Gamma\(|x-\xi+t\theta(x)-t\theta(\xi)|\)-\Gamma\(|x-\xi |\)\over t} \ &\hbox{if}\ x\in\partial\Omega  ,\\
\end{aligned}\right.
\end{equation}
where $\gamma^t$ is such that $I+\gamma^t=\(I+t\theta\)^{-1}$ so $\gamma^t(z)=-t\theta\(z+\gamma^t(z)\)$ and
\begin{equation}\label{6s}
f^t:={1\over t}\left\{\sum\limits_{i,j=1}^N{\partial^2\widetilde v^0_\xi   \over
 \partial x_i\partial x_j}  \[{\partial \gamma_j^t\over\partial z_i} +{\partial \gamma_i^t\over\partial z_j}+\sum\limits_{s=1}^N{\partial \gamma_j^t\over\partial z_s}{\partial \gamma_i^t\over\partial z_s}
 \] +\sum\limits_{j=1}^N{\partial  \widetilde v^0_\xi  \over
 \partial x_j}  \sum\limits_{s=1}^N{\partial ^2\gamma_j^t\over\partial z_s^2}\right\}.
\end{equation}
By the fact that $(I+t\theta)\circ(I+\gamma^t)=I$ and  by Remark \ref{pag2bis} we deduce that
$$\(\gamma^t\)'(z)[h]=-t(I+t\theta'(x))^{-1}\(\theta'(x)[h]\)$$
and
$$\(\gamma^t\)''(z)[h][k]=-t(I+t\theta'(x))^{-1}\(\theta''(x)\[h+\(\gamma^t\)'(z)[h]\]\)\[k+\(\gamma^t\)'(z)[k]\],$$
where $x=z+\gamma^t(z).$ Then we get that as $t\to0$
\begin{equation}\label{7s}
f^t\to-\sum\limits_{i,j=1}^N{\partial^2\widetilde v^0_\xi   \over
 \partial x_i\partial x_j}  \[{\partial \theta_j\over\partial x_i} +{\partial \theta_i\over\partial x_j}
 \] -\sum\limits_{j=1}^N{\partial  \widetilde v^0_\xi  \over
 \partial x_j}  \sum\limits_{s=1}^N{\partial ^2\theta_j \over\partial x_s^2} \ \hbox{in}\ C^{0,\alpha}(\overline\Omega),
\end{equation}
because $\|\widetilde v^\theta_\xi -\widetilde v^0_\xi \|_{C^{2,\alpha}(\overline\Omega)}\to0.$
Recalling that for $x\not=\xi$ we have
$$\lim\limits_{t\to0}{\Gamma\(|x-\xi+t\theta(x)-t\theta(\xi)|\)-\Gamma\(|x-\xi |\)\over t}=-\sum\limits_{i=1}^N
{x_i-\xi_i\over|x-\xi|^N}\(\theta_i(x)-\theta_i(\xi)\),$$
by \eqref{6s} and \eqref{7s}, using the standard regularity theory, we get that
 $\left\|{\widetilde v^{t\theta}_\xi-\widetilde v^0_\xi\over t}\right\|_{ C^{2,\alpha}(\overline\Omega)}$ is bounded.
Then for any sequence $(t_n)$ such that $t_n\to0$, the sequence of functions ${\widetilde v^{t_n\theta}_\xi-\widetilde v^0_\xi\over t_n},$ up to a subsequence, is convergent in  $C^{2}(\overline\Omega)$ and by \eqref{5s} and \eqref{7s} it
converges to the unique solution $u[\theta]$ of problem \eqref{1s}. In fact, by   Remark \ref{pag2} we have that
$$\left\|\sum\limits_{i,j=1}^N{\partial^2  {\widetilde v^\theta_\xi-\widetilde v^0_\xi\over t} \over
 \partial x_i\partial x_j}   \[{\partial \gamma_j^t\over\partial z_i} +{\partial \gamma_i^t\over\partial z_j}+\sum\limits_{s=1}^N{\partial \gamma_j^t\over\partial z_s}{\partial \gamma_i^t\over\partial z_s}
 \] +\sum\limits_{j=1}^N{\partial {\widetilde v^\theta_\xi-\widetilde v^0_\xi\over t}   \over
 \partial x_j}  \sum\limits_{s=1}^N{\partial ^2\gamma_j^t\over\partial z_s^2}\right\|_{ C^{0,\alpha}(\overline\Omega)}\to 0$$
 as $t\to0.$
Next, it is easy to check  that the Gateaux derivative exists and is continuous. Then the claim follows. \eqref{1s}.
\end{proof}

\begin{lemma}\label{lemma2} Let $p=1,\dots,N.$ For any $\xi\in\Omega $  the map $G:\mathfrak{B}_\rho\to C^{2,\alpha}(\overline\Omega)$ defined by $G(\theta)={\partial\widetilde v^\theta_\xi\over \partial x_p}$ is of class $C^1.$  Moreover
$$G'_\theta(0)[\theta]=D_\theta{\partial\widetilde v^\theta_\xi\over \partial x_p}|_{\theta=0}[\theta]=u_p[\theta]$$
is the unique solution of the problem
\begin{equation}\label{d7}
\left\{\begin{aligned}
 & \Delta_x u_p[\theta]-{\partial\over\partial x_p}\sum\limits_{i,j=1}^N{\partial^2 \widetilde v^0_\xi\over
 \partial x_i\partial x_j} \[{\partial \theta_j\over\partial x_i} +{\partial \theta_i\over\partial x_j}
 \]-{\partial\over\partial x_p}\sum\limits_{j=1}^N{\partial  \widetilde v^0_\xi\over
 \partial x_j} \Delta_x \theta_j  =0 \ &\hbox{if}\ x\in\Omega  \\
  & u_p[\theta](x)=-{\partial\over\partial x_p}\sum\limits_{i=1}^N {x_i-\xi_i\over |x-\xi|^N}\(\theta_i(x)-\theta_i(\xi)\) \ &\hbox{if}\ x\in\partial\Omega  .\\
\end{aligned}\right.
\end{equation}
\end{lemma}
\begin{proof}
First, we prove that the Gateaux derivative of the map $G$ at $0$ is the unique solution of the problem \eqref{1s}.
The function $ {w^{t\theta}_p-w^0_p\over t}$ is a solution of the problem
\begin{equation}\label{d8}
\left\{\begin{aligned}
 & \Delta_x \({w^{t\theta}_p-w^0_p\over t}\)+\sum\limits_{i,j=1}^N{\partial^2 \(w^{t\theta}_p-w^0_p\)\over
 \partial x_i\partial x_j}{1\over t} \[{\partial \gamma_j^t\over\partial z_i} +{\partial \gamma_i^t\over\partial z_j}+\sum\limits_{s=1}^N{\partial \gamma_i^t\over\partial z_s}{\partial \gamma_j^t\over\partial z_s}
 \] \\ & +\sum\limits_{j=1}^N{\partial \(w^{t\theta}_p-w^0_p\)\over
 \partial x_j}{1\over t}\sum\limits_{s=1}^N{\partial^2 \gamma_i^t\over\partial z_s^2} +{1\over t}\mathfrak{f}^{t\theta}(x)+{1\over t}\mathfrak{g}^{t\theta}(x) =0 \quad \hbox{if}\ x\in\Omega  \\
  & \({w^{t\theta}_p-w^0_p\over t}\)(x)={\varphi^{t\theta}_p(x)\over  t} \qquad \hbox{if}\ x\in\partial\Omega  ,\\
\end{aligned}\right.
\end{equation}
where  \begin{align}\label{d2}
&\mathfrak{f}^\theta(x):=\sum\limits_{i,j=1}^N{\partial^2 \widetilde v^\theta_\xi\over
 \partial x_i\partial x_j} {\partial\over\partial x_p}\[{\partial \gamma_j\over\partial z_i} +{\partial \gamma_i\over\partial z_j}+\sum\limits_{s=1}^N{\partial \gamma_i\over\partial z_s}{\partial \gamma_j\over\partial z_s}
 \] +\sum\limits_{j=1}^N{\partial \widetilde v^\theta_\xi\over
 \partial x_j} {\partial\over\partial x_p} \sum\limits_{s=1}^N{\partial^2 \gamma_i\over\partial z_s^2},
\\  \label{d3}
&\mathfrak{g}^\theta(x):=\sum\limits_{i,j=1}^N{\partial^2 w^0_p\over
 \partial x_i\partial x_j}   \[{\partial \gamma_j\over\partial z_i} +{\partial \gamma_i\over\partial z_j}+\sum\limits_{s=1}^N{\partial \gamma_i\over\partial z_s}{\partial \gamma_j\over\partial z_s}
 \] +\sum\limits_{j=1}^N{\partial w^0_p\over
 \partial x_j}   \sum\limits_{s=1}^N{\partial^2 \gamma_i\over\partial z_s^2},
\\ &\label{d4}
\varphi^\theta_p(x):={x_p-\xi_p+\theta_p(x)-\theta_p(\xi)\over |x -\xi +\theta  (x)-\theta (\xi)|^N}
-{x_p-\xi_p \over |x -\xi  |^N}+\sum\limits_{i=1}^N {x_i-\xi_i+\theta_i(x)-\theta_i(\xi)\over |x-\xi+\theta(x)-\theta(\xi)|^N} {\partial\theta_i\over\partial x_p}.
\end{align}
Moreover $\gamma^t$ is such that $I+\gamma^t=\(I+t\theta\)^{-1}$ so $\gamma^t(z)=-t\theta\(z+\gamma^t(z)\).$
 We point out that $\mathfrak{f}^{t\theta}$ in \eqref{d2}  and  $\mathfrak{g}^{t\theta}$ in \eqref{d3} also contain $\gamma^t.$
By Remark \ref{pag2bis} we deduce that if $z:=x+\theta(x)$
\begin{equation}\label{d9}
\gamma^t\(z\)=-t\theta(x)+\sum\limits_{i\ge2}(-1)^i\(t\theta(x)\)^i \ \hbox{and}\ (\gamma^t)'(z)=-t\theta'(x)+\sum\limits_{i\ge2}(-1)^i\(t\theta'(x)\)^i  .
\end{equation}
Then we have
\begin{equation}\label{d10}
\lim\limits_{t\to0}\left\|{1\over t}\mathfrak{f}^{t\theta}+{1\over t}\mathfrak{g}^{t\theta}
+
\sum\limits_{i,j=1}^N{\partial^2 \widetilde v^0_\xi\over
 \partial x_i\partial x_j}{\partial\over \partial x_p} \[{\partial \theta_j \over\partial x_i} +{\partial \theta_i \over\partial x_j}
 \] +\sum\limits_{j=1}^N{\partial \widetilde v^0_\xi\over
 \partial x_j}{\partial\over \partial x_p}  \Delta_x \theta_j
\right\|_{ C^{0,\alpha}(\overline\Omega)} =0.
\end{equation}
Since $ {w^{t\theta}_p-w^0_p\over t}$ solves problem \eqref{d8}, by estimate \eqref{d10} we deduce that $\left\|{w^{t\theta}_p-w^0_p\over t}\right\|_{ C^{2,\alpha}(\overline\Omega)} $ is bounded as $t\to0.$
Then for any sequence $(t_n)$ such that $t_n\to0$, the sequence of functions ${w^{t_n\theta}_p-w^0_p\over t_n},$ up to a subsequence, is convergent in  $C^{2}(\overline\Omega).$   Moreover, arguing as in Remark \ref{pag3} we can prove that
\begin{equation}\label{new1}\left\|{\partial \widetilde v^\theta_\xi\over\partial x_p}-{\partial \widetilde v^0_\xi\over\partial x_p} \right\|_{ C^{2,\alpha}(\overline\Omega)}\to 0\ \hbox{as}\ \|\theta\|_k\to0.\end{equation}
Finally, by \eqref{new1}, \eqref{d9} and \eqref{d10}, passing to the limit in \eqref{d8} as $t\to0$ we get the claim.
Next, it is easy to check  that the Gateaux derivative exists and is continuous. Then the claim follows.

\end{proof}

\end{document}